\documentclass[11pt,a4paper,reqno]{amsart}
\usepackage[T1]{fontenc}
\usepackage{lmodern,microtype,mathtools,amssymb}
\usepackage[margin=30mm]{geometry}
\usepackage[hidelinks]{hyperref}
\hypersetup{pdftitle={An induction proof of strong NF for amenable group C*-algebras},pdfauthor={Mehdi Moradi}}
\numberwithin{equation}{section}
\newtheorem{theorem}{Theorem}[section]
\newtheorem{lemma}[theorem]{Lemma}
\newtheorem{proposition}[theorem]{Proposition}
\theoremstyle{remark}
\newtheorem{remark}[theorem]{Remark}
\newcommand{\C}{\mathbb C}
\newcommand{\B}{\mathcal B}
\newcommand{\K}{\mathcal K}
\newcommand{\FC}{\operatorname{FC}}
\newcommand{\FCH}{\operatorname{FCH}}
\newcommand{\Ind}{\operatorname{Ind}}
\newcommand{\Ad}{\operatorname{Ad}}
\newcommand{\dist}{\operatorname{dist}}
\newcommand{\ext}{\partial_{\mathrm e}}
\newcommand{\norm}[1]{\lVert#1\rVert}
\title[An induction proof of strong NF]{An induction proof of strong NF\ for amenable group $C^*$-algebras}
\author{Mehdi Moradi}
\address{University of Toronto, Toronto, Canada}
\email{mehdi.moradi@utoronto.ca}
\date{September 2026}

\begin{document}
\begin{abstract}
We give a proof that the full group $C^*$-algebra of a countable discrete amenable group is strong NF. The representation-theoretic construction uses the exact sequence whose kernel is the hyper-FC radical and whose quotient is the maximal ICC quotient. Extreme invariant traces on the radical induce factorial representations of the group. A Choquet decomposition of the regular radical trace selects a separating family whose induced traces are supported on the ordinary FC-centre. Exact twisted crossed-product quotients and a sequential UCT argument make their images quasidiagonal. Faithful irreducible replacements then supply the projections required for inner quasidiagonality. We also give a direct trace-decomposition variant and identify the steps shared by the two proofs.
\end{abstract}
\maketitle

\section{Introduction and the ICC reduction}

\begin{theorem}\label{thm:main}
Let $G$ be a countable discrete amenable group. Then its full complex group $C^*$-algebra $C^*(G)$ is strong NF in the sense of Blackadar--Kirchberg.
\end{theorem}

\subsection*{Rosenberg's conjecture and stronger approximation properties}
In his 1987 appendix to Hadwin's paper, Rosenberg proved that a countable discrete group whose left regular representation is quasidiagonal is amenable \cite[Theorem~A1]{Ros87}. The corresponding implication for quasidiagonality of $C_r^*(G)$, and its converse, place group amenability in an operator-norm approximation framework. The converse became known as Rosenberg's conjecture. Progress included the amenable LEF case, recovered through the MF-group characterization of Carri\'on--Dadarlat--Eckhardt \cite[Theorem~2.8 and Proposition~2.13]{CDE13}, whose introduction credits Kerr with an earlier proof, and the elementary amenable case of Ozawa--R\o rdam--Sato \cite[Theorem~3.8(iv)]{ORS15}. Here LEF means locally embeddable into finite groups; the amenability hypothesis in the first result is essential.

Tikuisis--White--Winter resolved the conjecture in \cite[Corollary~C]{TWW17}. Their Theorem~A proves quasidiagonality of faithful traces on separable nuclear UCT algebras, and Corollary~B gives quasidiagonality of the algebra when such a trace exists. For amenable groups, the canonical trace is faithful and Tu's theorem supplies the UCT \cite[Lemma~3.5 and Proposition~10.7]{Tu99}. Thus every countable amenable group algebra is quasidiagonal. The present theorem strengthens this conclusion to inner quasidiagonality and hence strong NF. This requires finite-rank projections in the represented bicommutant, as constructed in Section~\ref{sec:replacements}.

Strong NF must also be distinguished from \emph{strong quasidiagonality}, which requires every representation to be quasidiagonal. The lamplighter group
\[
L=(\mathbb Z/2\mathbb Z)\wr\mathbb Z
 =\Bigl(\bigoplus_{\mathbb Z}\mathbb Z/2\mathbb Z\Bigr)\rtimes\mathbb Z
\]
illustrates the distinction. Carri\'on--Dadarlat--Eckhardt construct a nonfinite quotient of $C^*(L)$, and hence a non-quasidiagonal representation image \cite[Theorem~3.4 and Corollary~3.5]{CDE13}. Their construction uses an infinite projection. They also note that $C^*(L)$ is residually finite dimensional and inner quasidiagonal. Since $L$ is amenable, this algebra is strong NF. A separating family of quasidiagonal irreducible representations is compatible with other representation images failing even to be finite; our proof selects such a family and does not assert that all representations are quasidiagonal.

On the positive side, Eckhardt proved strong quasidiagonality for finitely generated solvable virtually nilpotent groups \cite[Theorem~1.9]{Eck14}; Eckhardt--Gillaspy--McKenney removed the solvability restriction and treated inductive limits of virtually nilpotent groups \cite[Corollary~4.6]{EGM19}. Eckhardt also exhibited a non-virtually-nilpotent group $\mathbb Z^3\rtimes\mathbb Z^2$ with strongly quasidiagonal group algebra, whereas for $\mathbb Z^d\rtimes\mathbb Z$ strong quasidiagonality is equivalent to virtual nilpotence \cite[Theorems~2.4 and~3.3]{Eck15}. These results describe the stronger all-representations property. His later work proves classification results for irreducible images of finitely generated virtually nilpotent groups \cite[Theorem~A]{Eck24}. For the present proof the relevant input from that work is Lemma~2.5: twisted crossed products of separable type~I algebras by countable amenable discrete groups satisfy the UCT. We apply it only to the specified type~I coefficient stages in Section~\ref{sec:uct}.

\subsection*{Why the UCT predicts strong NF, and why the quotients matter}
The UCT conjecture asks whether every separable nuclear $C^*$-algebra satisfies the universal coefficient theorem of Rosenberg--Schochet \cite{RS87}. It predicts the strong-NF conclusion here, not just the already established quasidiagonality of the group algebra. More precisely, it has the following conditional consequence.

\begin{remark}[The UCT prediction]\label{rem:uct-prediction}
If every separable nuclear $C^*$-algebra satisfies the UCT, then every separable unital nuclear $C^*$-algebra $D$ with a faithful tracial state is strong NF. To see this, represent a faithful trace as a Choquet integral of extreme traces on $D$. Their GNS representations separate $D$: an element annihilated by all of them has zero faithful trace on its square modulus. For each extreme trace $t$, the central-projection argument in Proposition~\ref{prop:direct} shows that $\pi_t(D)''$ is a factor; that argument uses only extremality of a trace. The image $D/\ker\pi_t$ is separable and nuclear and carries a faithful descended trace. Under the stated UCT assumption it is quasidiagonal. The factor-to-prime, faithful-irreducible-replacement and compact-exclusion arguments of Section~\ref{sec:replacements} therefore give a faithful quasidiagonal irreducible representation of each image. Their compositions with the quotient maps separate $D$. The finite norm-witness and bicommutant-projection construction in that section then proves inner quasidiagonality of $D$, and nuclearity gives strong NF. Every hypothesis of those shared arguments is present for these images.
\end{remark}

Our unconditional proof establishes the UCT for the particular trace quotients it needs. Although Tu's theorem already gives the UCT for $C^*(G)$, it does not supply the UCT for arbitrary representation images. For a tracial state $t$ on $A=C^*(G)$, the appropriate quotient is
\[
B_t=A/K_t,\qquad K_t=\ker\pi_t=\{a\in A:t(a^*a)=0\}.
\]
The descended trace on $B_t$ is faithful, and $B_t$ remains separable and nuclear. Consequently UCT for $B_t$ is exactly the remaining hypothesis needed for our application of \cite[Corollary~B]{TWW17}. Quasidiagonality of $A$ alone does not imply quasidiagonality of its quotients, as the lamplighter example demonstrates.

Faithfulness here is a hypothesis of this algebra-level argument. Tikuisis--White--Winter explicitly discuss the quotient-UCT issue in \cite[Remark~5.6]{TWW17}; their Corollary~6.1 also gives quasidiagonality of \emph{all} traces on a separable unital nuclear quasidiagonal UCT algebra. That latter conclusion does not by itself give quasidiagonality of $A/K_t$: matrix approximations to a trace on $A$ need not annihilate $K_t$ in operator norm and descend to the quotient. Passing to the actual GNS image removes the trace-null ideal, while proving UCT there permits a new application of the faithful-trace theorem. The extra work in Section~\ref{sec:uct} is therefore at the level of the chosen images: FC support identifies their kernels as generated restriction ideals, and type~I twisted quotient stages give the required sequential UCT limit. No universal UCT assumption enters Theorem~\ref{thm:main}.

\subsection*{The induction construction}
For ICC groups, Theorem~\ref{thm:main} is Amini--Moradi's Proposition~D \cite{AM23}. We organize the general proof around the hyper-FC radical and induction. The factorial-image argument below extends the representation-theoretic step in their ICC proof. The UCT input for the induced images is established separately; the untwisted ICC theorem is not applied to an arbitrary crossed product.

Section~\ref{sec:direct} gives a second construction of the separating factorial images, using extreme traces of the whole group algebra. It shares the UCT and irreducible-representation arguments of Sections~\ref{sec:uct} and~\ref{sec:replacements} with the induction proof.

All groups below are discrete. Write $u_g$ for the canonical group unitaries. Amenability gives $C^*(G)=C_r^*(G)$ and nuclearity, and countability gives separability. We use the faithful regular trace $\tau_G$ and the canonical faithful subgroup expectations on these reduced algebras. In particular the full subgroup algebras embed isometrically. A separable $C^*$-algebra is strong NF precisely when it is nuclear and inner quasidiagonal \cite[Theorem~4.5]{BK01}.

Put $F=\FC(G)$. Define the upper FC series by
\[
R_0=\{e\},\qquad R_{\alpha+1}/R_\alpha=\FC(G/R_\alpha),
\qquad R_\lambda=\bigcup_{\alpha<\lambda}R_\alpha
\]
at limit ordinals. Its stabilized value is the hyper-FC radical $R=\FCH(G)$; see \cite[Section~2.4]{GU25}. Each term is characteristic. For countable $G$ the series stabilizes at a countable ordinal: if every successor below $\omega_1$ were strict, choosing a new element at each such stage would contradict countability. Equality at one successor makes all later terms equal. Thus
\begin{equation}\label{eq:extension}
1\longrightarrow R\longrightarrow G\overset{p}{\longrightarrow}Q\longrightarrow1,
\qquad \FC(Q)=\{e\}.
\end{equation}
We allow the trivial group in the ICC convention. The quotient in \eqref{eq:extension} is maximal: every surjection from $G$ onto an ICC group kills each $R_\alpha$, by transfinite induction. At a successor, its induced surjection sends finite conjugacy classes to finite conjugacy classes, which must be trivial; at a limit it kills the union. It therefore factors through $Q$. Notice that $F\subseteq R$, without any assertion that $F=R$.

\section{Invariant traces and induced factors}

Write $A=C^*(G)$, $C=C^*(R)$ and
\[
\mathcal T=T(C)^G,
\]
the compact metrizable convex set of tracial states invariant under conjugation by $G$. It contains $\tau_R$. For $\sigma\in\mathcal T$ set
\begin{equation}\label{eq:indtrace}
t_\sigma=\sigma\circ E_R,\qquad
\Pi_\sigma=\pi_{t_\sigma},\qquad B_\sigma=\Pi_\sigma(A),
\end{equation}
where $E_R:A\to C$ is the subgroup expectation. Positivity is immediate. The functional $t_\sigma$ is tracial: for $g,h\in G$, either both $gh,hg$ lie outside $R$, or both lie in $R$ and are $G$-conjugate. Their values agree in either case, and density of the group span proves traciality.

\begin{lemma}\label{lem:induced-factor}
For every $\sigma\in\mathcal T$,
\[
\Pi_\sigma\cong\Ind_R^G(\pi_\sigma).
\]
If $\sigma\in\ext\mathcal T$, then $\Pi_\sigma(A)''$ is a finite factor.
\end{lemma}

\begin{proof}
Let $(\pi_\sigma,H_\sigma,\xi_\sigma)$ be the tracial GNS representation of $C$. The induced representation acts by left translation on square-summable functions $f:G\to H_\sigma$ satisfying
\[
f(xr)=\pi_\sigma(r^{-1})f(x)\qquad(r\in R),
\]
with norm squared summed over the cosets $G/R$. The vector $\Omega$ supported on $R$, with $\Omega(r)=\pi_\sigma(r^{-1})\xi_\sigma$, has coefficient $t_\sigma$. It is cyclic: for a section $s:Q\to G$, the vectors obtained by translating $\Omega$ by $s(q)r$ have, in the $q$-coordinate, the dense family $\pi_\sigma(r)\xi_\sigma$. GNS uniqueness proves the claimed equivalence.

Put $M=\Pi_\sigma(A)''$ and $N=\Pi_\sigma(C)''$. The GNS trace extends to a faithful normal trace $T$ on $M$: bounded right multiplication commutes with left multiplication and has the same cyclic vector, making $\Omega$ separating. The restriction of $N$ to $\overline{N\Omega}$ is faithful, since an element annihilating that subspace annihilates $\Omega$. It identifies $N$ with $\pi_\sigma(C)''$, with its faithful normal trace.

Choose $s(e)=e$ and put $V_q=\Pi_\sigma(u_{s(q)})$. The extension need not split:
\[
V_qV_r=w(q,r)V_{qr},\qquad
w(q,r)=\Pi_\sigma(u_{s(q)s(r)s(qr)^{-1}})\in N.
\]
Nevertheless the coset decomposition gives the exact orthogonal sum
\begin{equation}\label{eq:grading}
L^2(M,T)=\bigoplus_{q\in Q}\mathcal H_q,
\qquad \mathcal H_q=\overline{NV_q\Omega}.
\end{equation}
Orthogonality follows from vanishing of $t_\sigma$ off $R$, and density from the group span. Conjugation by $\Pi_\sigma(u_g)$ acts unitarily on $L^2(M,T)$ and sends $\mathcal H_q$ onto $\mathcal H_{p(g)qp(g)^{-1}}$.

For $z\in Z(M)$ the vector $z\Omega$ is fixed by these conjugations. Its component norms in \eqref{eq:grading} are constant on $Q$-conjugacy classes. Their squares are summable, so ICC forces every nonidentity component to vanish. The trace-preserving normal expectation $E_N:M\to N$ is the $L^2$ projection onto $\mathcal H_e$. Therefore
\[
z\Omega=E_N(z)\Omega,\qquad z=E_N(z)\in N.
\]
Commutation with $N$ and all the $V_q$ now gives
\begin{equation}\label{eq:center}
Z(M)=Z(N)^G.
\end{equation}
The action on $Z(N)$ factors through $Q$, since $R$ acts by inner automorphisms. This argument retains the operator-valued cocycle; no untwisting assumption is used.

Suppose that $\sigma$ is extreme in $\mathcal T$ and that $Z(N)^G\ne\C$. A nonscalar self-adjoint invariant central element has a spectral projection $e$ with $0<e<1$. Put $c=T(e)\in(0,1)$. The functionals on $C$ given by
\[
a\longmapsto c^{-1}T(e\Pi_\sigma(a)),\qquad
a\longmapsto(1-c)^{-1}T((1-e)\Pi_\sigma(a))
\]
are invariant tracial states and decompose $\sigma$ with weights $c,1-c$. They are distinct: equality on the weakly dense image of $C$ would give equality of their normal extensions to $N$, contradicted by evaluation at $e$. Extremality excludes this. Thus \eqref{eq:center} is scalar and $M$ is a factor. The proof also applies when $Q=\{e\}$.
\end{proof}

\section{Selecting a separating family}

\begin{proposition}\label{prop:selection}
There are a Borel probability measure $\mu$ on $\mathcal T$, with barycentre $\tau_R$ and carried by $\ext\mathcal T$, and a conull Borel set $X\subseteq\ext\mathcal T$ such that $t_\sigma$ vanishes outside $F$ for every $\sigma\in X$. The induced representations $(\Pi_\sigma)_{\sigma\in X}$ separate $A$.
\end{proposition}

\begin{proof}
Choose a representing measure by the metrizable Choquet theorem \cite[Theorem~4.2]{BD59}. Fix $r\in R\setminus F$. For every $m\geq1$ choose distinct $G$-conjugates $r_1,\ldots,r_m$, all in $R$, and set $b_m=m^{-1}\sum_j u_{r_j}$. Every $\sigma\in\mathcal T$ satisfies $\sigma(b_m)=\sigma(u_r)$. Cauchy--Schwarz and the barycentre identity give
\begin{equation}\label{eq:square-bound}
\int_{\mathcal T}|\sigma(u_r)|^2\,d\mu(\sigma)
\leq\int_{\mathcal T}\sigma(b_m^*b_m)\,d\mu(\sigma)
=\tau_R(b_m^*b_m)=\frac1m.
\end{equation}
Hence $\sigma(u_r)=0$ almost everywhere. The extreme boundary is Borel, each evaluation zero set is closed, and $R$ is countable. Consequently
\[
X=\ext\mathcal T\cap\bigcap_{r\in R\setminus F}
\{\sigma\in\mathcal T:\sigma(u_r)=0\}
\]
is Borel and conull. Extension by zero outside $R$ shows that $t_\sigma$ vanishes outside $F$ on this set.

Since $\tau_R E_R=\tau_G$, for every $a\in A$ we have
\begin{equation}\label{eq:barycentre}
\tau_G(a)=\int_X t_\sigma(a)\,d\mu(\sigma).
\end{equation}
If all $\Pi_\sigma$ annihilate $a$, then \eqref{eq:barycentre} applied to $a^*a$ gives $\tau_G(a^*a)=0$. Faithfulness yields $a=0$. Thus the product homomorphism is injective and isometric, and
\begin{equation}\label{eq:norm}
\norm{a}=\sup_{\sigma\in X}\norm{\Pi_\sigma(a)}.
\end{equation}
\end{proof}

\section{UCT for the selected induced images}\label{sec:uct}

We prove the required $C^*$-level quotient and limit assertions. For a tracial state $t$ on $A$, write
\[
K_t=\ker\pi_t=\{a\in A:t(a^*a)=0\}.
\]
The trace-null space is a closed two-sided $*$-ideal. Besides Cauchy--Schwarz and $t(a^*a)=t(aa^*)$, the relevant multiplication estimates are
\[
t((ba)^*(ba))\leq\norm{b}^2t(a^*a),\qquad
t((ab)^*(ab))=t(abb^*a^*)\leq\norm{b}^2t(aa^*).
\]
This ideal annihilates every GNS vector, and the converse follows at the cyclic vector. The trace descended to $A/K_t$ is therefore faithful.

\begin{lemma}\label{lem:kernel}
Let $L\lhd G$, $J_L=K_t\cap C^*(L)$, $D_L=C^*(L)/J_L$, and let $I_L$ be the closed ideal of $A$ generated by $J_L$. A normalized section of $G/L$ determines a twisted action with
\begin{equation}\label{eq:quotient}
A/I_L\cong D_L\rtimes_{\alpha,w}(G/L).
\end{equation}
The full and reduced products agree. If $t(u_g)=0$ for every $g\notin L$, then $K_t=I_L$.
\end{lemma}

\begin{proof}
The ideal $J_L$ is invariant under conjugation by $G$. Write $v_l=u_l+J_L$, choose a section $s:G/L\to G$ with $s(e)=e$, and set
\[
c(q,r)=s(q)s(r)s(qr)^{-1},\qquad
\alpha_q(v_l)=v_{s(q)ls(q)^{-1}},\qquad w(q,r)=v_{c(q,r)}.
\]
Invariance makes the automorphisms well defined. The identities
\[
\alpha_q\alpha_r=\Ad(w(q,r))\alpha_{qr},\qquad
\alpha_q(w(r,h))w(q,rh)=w(q,r)w(qr,h)
\]
follow by expanding products of the section elements; the second corresponds to
\[
s(q)c(r,h)s(q)^{-1}c(q,rh)=c(q,r)c(qr,h).
\]
Normalization follows from $s(e)=e$.

We have $I_L\subseteq K_t$, whence $I_L\cap C^*(L)=J_L$. Thus $A/I_L$ contains a faithful copy of $D_L$. This copy and the section unitaries give a covariant representation and a surjection from the full twisted product to $A/I_L$. Conversely, for the canonical crossed-product unitaries $U_q$, the assignment
\[
ls(q)\longmapsto v_lU_q
\]
is a group representation, since
\[
(v_lU_q)(v_kU_r)=v_{l s(q)k s(q)^{-1}c(q,r)}U_{qr}.
\]
It integrates to a homomorphism annihilating $J_L$ and hence $I_L$. The two homomorphisms are inverse on generators. This proves \eqref{eq:quotient}. Amenability of $G/L$ identifies the product with its reduced version, with faithful coefficient expectation $E_D$ \cite{PR89}.

If $t$ is supported on $L$, its restriction induces a faithful trace $\bar t_L$ on $D_L$. The natural map $r:A/I_L\to A/K_t$ satisfies
\[
\bar t\circ r=\bar t_L\circ E_D,
\]
as is checked on group unitaries inside and outside $L$, then by density. The right side is faithful: for $x\geq0$, a zero value forces first $E_D(x)=0$ and then $x=0$. Applying this to $y^*y$ for $y\in\ker r$ proves that $r$ is injective, hence $K_t=I_L$.
\end{proof}

In particular, Lemma~\ref{lem:kernel} applied to \eqref{eq:indtrace} gives the actual induced $C^*$-image identity
\begin{equation}\label{eq:radical-image}
B_\sigma\cong(C^*(R)/\ker\pi_\sigma)\rtimes_{\alpha,w}Q.
\end{equation}
We do not impose a type I assumption on the coefficient in \eqref{eq:radical-image}.

\begin{proposition}\label{prop:uct}
If $t$ is a tracial state of $A$ vanishing on $G\setminus F$, then $\pi_t(A)$ satisfies the UCT. In particular, this holds for every $B_\sigma$, $\sigma\in X$.
\end{proposition}

\begin{proof}
Enumerate $F$ as $f_1,f_2,\ldots$, allowing repetitions when it is finite. Let $F_n$ be generated by the union of the $G$-conjugacy classes of $f_1,\ldots,f_n$. This is a finite generating set invariant under $G$, so $F_n\lhd G$, and the $F_n$ increase to $F$. Each generating element has finite conjugacy class in $F_n$. Intersecting its finite-index centralizers over the finite generating set shows
\[
[F_n:Z(F_n)]<\infty.
\]
In an irreducible representation of $F_n$, all central unitaries are scalar. If $m_n=[F_n:Z(F_n)]$, the image algebra is consequently spanned by at most $m_n$ central coset representatives. It is finite dimensional; irreducibility makes it a full matrix algebra. The resulting uniform bound at each fixed $n$ proves that $C^*(F_n)$ is subhomogeneous and hence type I.

Set $J=K_t\cap C^*(F)$, $J_n=J\cap C^*(F_n)$, and let $I,I_n$ be the respective closed generated ideals in $A$. By Lemma~\ref{lem:kernel}, $I=K_t$ and
\begin{equation}\label{eq:stages}
A/I_n\cong(C^*(F_n)/J_n)\rtimes_{\alpha^{(n)},w^{(n)}}(G/F_n).
\end{equation}
Each coefficient is separable and type I, and each acting group is countable amenable. Eckhardt's type I coefficient theorem \cite[Lemma~2.5]{Eck24} proves UCT for every algebra in \eqref{eq:stages}. These algebras are also separable and nuclear, being quotients of $A$.

The following ideal equalities identify the limit:
\begin{equation}\label{eq:ideal-limit}
J=\overline{\bigcup_n J_n},\qquad
I=\overline{\bigcup_n I_n}.
\end{equation}
For the first, take $a\in J$ and $\delta>0$, and choose $b\in C^*(F_n)$ with $\norm{a-b}<\delta$. The induced map $C^*(F_n)/J_n\to C^*(F)/J$ is injective and isometric, so
\[
\dist(b,J_n)=\norm{b+J}<\delta.
\]
Choose $c\in J_n$ with $\norm{b-c}<2\delta$, giving $\norm{a-c}<3\delta$. For the second equality, the increasing union closure of the ideals $I_n$ is an ideal containing all $J_n$, hence $J$ and $I$; the reverse inclusion is immediate.

Thus
\begin{equation}\label{eq:uct-limit}
\pi_t(A)\cong A/I\cong\varinjlim_n A/I_n,
\end{equation}
where the connecting maps are the canonical surjections. Explicitly, for every $a\in A$ and every tail of the system,
\[
\lim_{n\to\infty}\norm{a+I_n}
=\lim_{n\to\infty}\dist(a,I_n)
=\dist(a,I)=\norm{a+I}.
\]
This proves the inductive-limit identification, including its norm. The separable nuclear UCT class is closed under sequential limits with possibly noninjective connecting maps \cite[Definition~22.3.4(N2), Theorem~23.1.1 and Section~3.3]{Blackadar98}. Applying this to \eqref{eq:uct-limit} proves the claim. No UCT permanence under arbitrary quotients is used.
\end{proof}

\section{Irreducible replacements and inner quasidiagonality}\label{sec:replacements}

\begin{proof}[Proof of Theorem~\ref{thm:main}]
For $\sigma\in X$, the algebra $B_\sigma$ is separable, nuclear, UCT, and has a faithful trace. It is quasidiagonal by Tikuisis--White--Winter \cite[Corollary~B]{TWW17}. Lemma~\ref{lem:induced-factor} gives a faithful factorial representation of $B_\sigma$, so $B_\sigma$ is prime. Indeed, the support projection of any nonzero ideal belongs to the center of this factor and is nonzero, hence is $1$; its approximate identity converges strongly to $1$. Two nonzero ideals cannot therefore have zero product.

By separability and Dixmier's theorem \cite[Corollaire~1]{Dix60}, choose a faithful irreducible representation $\rho_\sigma:B_\sigma\to\B(H_\sigma')$. Its Hilbert space is separable because the representation is cyclic. If it is finite dimensional, it is quasidiagonal. Otherwise its image meets $\K(H_\sigma')$ only in zero. To see this, a nonzero compact in an irreducible image forces containment of all compacts \cite[Corollary~IV.1.2.5]{Blackadar06}. The faithful trace transported to that image would assign the same positive value $c$ to every rank-one projection, since they are equivalent within the compact algebra. Arbitrarily large orthogonal families would give $mc\leq1$ for every $m$, a contradiction.

The faithful essential representation theorem of Voiculescu now makes $\rho_\sigma$ quasidiagonal \cite{Voi76}; see also \cite[proof of Proposition~2.5]{BK01}. Thus
\[
\theta_\sigma=\rho_\sigma\circ\Pi_\sigma
\]
is an irreducible quasidiagonal representation of $A$. Faithfulness of $\rho_\sigma$ preserves the kernel of $\Pi_\sigma$, so the family separates and retains \eqref{eq:norm} with $\theta_\sigma$ in place of $\Pi_\sigma$.

Let $E\subseteq A$ be a nonempty finite set and $\varepsilon>0$. Choose finitely many of these irreducibles detecting each $\norm{a}$, $a\in E$, within $\varepsilon/3$. Remove unitary-equivalence duplicates and call the resulting representations $\theta_1,\ldots,\theta_m$. Their norms still detect every required element. Schur's lemma gives
\[
\theta(A)''=\bigoplus_{j=1}^m\B(H_j),\qquad
\theta=\bigoplus_{j=1}^m\theta_j.
\]
Choose a finite-rank quasidiagonal projection $p_j$ in each summand such that, for every $a\in E$,
\[
\norm{[p_j,\theta_j(a)]}<\varepsilon,
\qquad
\norm{p_j\theta_j(a)p_j}>\norm{\theta_j(a)}-\varepsilon/3.
\]
The compression condition follows from strong convergence to $1$, applied to a unit vector nearly attaining the operator norm and its image. In finite dimension take $p_j=1$. The finite-rank projection $p=\bigoplus_j p_j$ lies in $\theta(A)''$ and satisfies
\[
\norm{[p,\theta(a)]}<\varepsilon,
\qquad
\norm{p\theta(a)p}>\norm{a}-2\varepsilon/3>\norm{a}-\varepsilon
\quad(a\in E).
\]
The empty set has no conditions and permits $p=0$. These are precisely the inner quasidiagonality inequalities \cite[Definition~2.2]{BK01}. Nuclearity now proves strong NF.
\end{proof}

\begin{remark}
The roles of the two radicals are distinct. The hyper-FC radical supplies the ICC quotient and the induced factors; the ordinary FC-centre supplies the type I stages after the conull selection. The proof uses no finite-orbit hypothesis on irreducible representations of the radical. The quasidiagonal irreducible replacements are representations of the induced $C^*$-images, and are not asserted to be subrepresentations of the induced factors. Finite and trivial groups, finite FC-centres, and the case $R=G$ are covered by the same construction.
\end{remark}

\section{Direct trace decomposition and comparison}\label{sec:direct}

The separating factorial images can also be obtained by decomposing the regular trace of the whole group algebra. We give the construction explicitly so that its relationship with the induction proof is precise.

\begin{proposition}\label{prop:direct}
There is a Borel probability measure $\nu$ on $T(A)$, with barycentre $\tau_G$ and carried by $\ext T(A)$, and a conull Borel set $Y\subseteq\ext T(A)$ such that every $t\in Y$ vanishes on $G\setminus F$. The representations $(\pi_t)_{t\in Y}$ separate $A$, and each $\pi_t(A)''$ is a finite factor.
\end{proposition}

\begin{proof}
Apply the metrizable Choquet theorem to $\tau_G\in T(A)$. For $g\in G\setminus F$ and any $m\geq1$, choose distinct conjugates $g_1,\ldots,g_m$ of $g$ and set $b_m=m^{-1}\sum_j u_{g_j}$. Traciality gives $t(b_m)=t(u_g)$ for every $t\in T(A)$. Thus
\[
\int_{T(A)}|t(u_g)|^2\,d\nu(t)
\leq\int_{T(A)}t(b_m^*b_m)\,d\nu(t)
=\tau_G(b_m^*b_m)=\frac1m.
\]
It follows that $t(u_g)=0$ almost everywhere. Countability of $G$ and Borelness of the extreme boundary show that
\[
Y=\ext T(A)\cap\bigcap_{g\in G\setminus F}
\{t\in T(A):t(u_g)=0\}
\]
is Borel and conull. If $\pi_t(a)=0$ for every $t\in Y$, then
\[
\tau_G(a^*a)=\int_Y t(a^*a)\,d\nu(t)=0,
\]
so faithfulness of $\tau_G$ gives $a=0$. The corresponding product representation is isometric, and hence
\begin{equation}\label{eq:direct-norm}
\norm{a}=\sup_{t\in Y}\norm{\pi_t(a)}\qquad(a\in A).
\end{equation}

For factoriality, fix $t\in\ext T(A)$ and write $M_t=\pi_t(A)''$, with its faithful normal GNS trace $T_t$. If $Z(M_t)\ne\C$, a central spectral projection $e$ satisfies $0<e<1$. Let $c=T_t(e)\in(0,1)$. The two states
\[
a\longmapsto c^{-1}T_t(e\pi_t(a)),\qquad
a\longmapsto(1-c)^{-1}T_t((1-e)\pi_t(a))
\]
are tracial and decompose $t$ with weights $c,1-c$. They are distinct: equality on $\pi_t(A)$ would extend by normality and weak density to $M_t$, where their values at $e$ are $1$ and $0$. This contradicts extremality. Thus $M_t$ is a finite factor, in particular for every $t\in Y$.
\end{proof}

For $t\in Y$, put $B_t=\pi_t(A)$. Proposition~\ref{prop:uct} gives the UCT for $B_t$, and its descended trace is faithful. It is separable and nuclear as a quotient of $A$, so Tikuisis--White--Winter makes it quasidiagonal. Proposition~\ref{prop:direct} supplies its faithful factorial representation. The prime-to-primitive argument in Section~\ref{sec:replacements} therefore gives a faithful irreducible representation of $B_t$; the faithful trace excludes compact operators in the infinite-dimensional case, and Voiculescu's theorem makes that representation quasidiagonal. Composition with $\pi_t$ preserves its kernel and the norm identity \eqref{eq:direct-norm}. Applying the finite norm-witness and bicommutant-projection construction of Section~\ref{sec:replacements} to this family proves inner quasidiagonality and hence Theorem~\ref{thm:main} again.

The difference is in the space being decomposed and the source of factoriality. The induction proof decomposes $\tau_R$ in $T(C^*(R))^G$ and uses the ICC quotient to prove factoriality after induction. The direct proof decomposes $\tau_G$ in $T(A)$, where extremality itself gives factoriality. Both use the same conjugacy-averaging estimate to obtain FC support, the same type I twisted quotient stages and UCT limit, and the same faithful irreducible replacements and finite-rank projections. The direct construction shortens the selection and factoriality argument; the induction construction exhibits how the maximal ICC quotient produces the factors. This comparison concerns the two constructions given here and makes no identification of their representing measures or selected families.

\section*{Acknowledgements}
A proof of Theorem~\ref{thm:main} was sketched by the author in his talk at COSy 2024 at the University of Alberta. OpenAI's ChatGPT/Codex helped check references and prepare the exposition and \LaTeX{} source.
The author gratefully acknowledges support from the Fields Institute and the University of Ottawa during his postdoctoral fellowship from July 2023 to June of 2025.
He thanks Chris Schafhauser for discussions regarding this work during the operator algebra thematic program.

\end{document}